\documentclass[12pt, twoside, leqno]{article}
\usepackage{amsmath,amsthm}
\usepackage{amssymb}
\usepackage{enumitem}
\usepackage{graphicx}
\newtheorem{theorem}{Theorem}[section]

\newtheorem{lemma}[theorem]{Lemma}
\newtheorem{proposition}[theorem]{Proposition}

\theoremstyle{definition}
\newtheorem{definition}[theorem]{Definition}

\newtheorem{example}[theorem]{Example}

\numberwithin{equation}{section}

\begin{document}

\title{Homology of weighted path complexes and directed hypergraphs}

\author{ Yuri Muranov,  
Anna Szczepkowska, 
and 
Vladimir Vershinin
}

\date{}

\maketitle

\begin{abstract}
We introduce  the   weighted path homology  on the category of weigh\-ted directed hypergraphs
 and describe  conditions 
of homotopy invariance of weighted path homology groups. We give several examples that 
explain the nontriviality of the introduced notions. 
\end{abstract}



\section{Introduction}\setcounter{equation}{0}\label{S1}

One of the  tools for the investigation of  global structures in graph theory 
are  (co)homology  theories which were constructed recently by many authors for various  
categories  of (di)graphs, multigraphs, quivers, and  hypergraphs  \cite{Talbi},  \cite{Forum}, \cite{Parks}, 
 \cite{Graham}, \cite{Betti}, \cite{Hyper},   \cite{Embed}, and \cite{Miarx}.  

The notion of directed hypergraphs arises in 
discrete mathematics as a  natural generalization of digraphs and hypergraphs \cite{Berge} and   \cite{Gallo}.   
 At present time there are many mathematical models of various   complex systems  in which the system is presented by a network  whose interacting pairs of nodes are connected by links.  Such models are naturally presented by hypergraphs with additional structures \cite{Bat}. The notions of path, cycle and weight structure on hypergraphs    arise naturally for the  directed hypergraphs. These consepts  allow to investigate  global structures on hypergraphs  by methods of algebraic topology and, in particular, homology theory. 
The excellent  survey of application  of directed hypergraphs to various aspects of  computer sciences is given in paper  \cite{Sur}. 
 The application of  homology  theory of hypergraphs in the  pharmaceutical industries is described, for example,  in \cite{Cang}, \cite{Liu}.

In the present paper,  we introduce several categories  of weighted directed hypergraphs and the notion 
of homotopy in these categories.  Then we define
 functorial homology theories  on such categories
 using the  homology theory of path complexes  constructed in  \cite{Hyper},      
 \cite{Pcomplex},  and \cite{Forum}. We describe also the conditions of homotopy 
 invariance of introduced homology groups 
and we give several examples.  In the paper we consider only finite 
digraphs, path complexes and directed hypergraphs.

 \section{Path homology of weighted path complexes}\label{S2}
\setcounter{equation}{0}

 A \emph{path complex}  \cite[Sec. 3.1]{Pcomplex}  $\Pi=(V, P)$  consists of a set  
 $V$  of vertices and a set $P$  of \emph{elementary paths}    $ (i_{0}\dots i_n)$ 
 of vertices   such that  one-element sequence $(i)$ lies in $P$ and  if
  $(i_{0}\dots i_{n})\in P$   then
$(i_{0}...i_{n-1})\in P$,  $(i_{1}...i_{n})\in P$.  
The  \emph{length} of a path 
$(i_{0}...i_{n})$  is equal to $n$. We denote by $ \Pi_V=(V, P_V)$ the path complex 
for which $P_V$ consists of all paths of finite length on $V$.

 A \emph{morphism of path complexes} $f\colon (V,P)\to (W, Q)$ is given by the pair 
 of maps $(f_V,f_P)$ where $f_V\colon V\to W$ 
 and    $f_P(i_{0}\dots i_{n})\colon =\left(f_V(i_{0})\dots f_V(i_{n})\right)$  lays in  $P$.   
 Thus we obtain a  category  $\mathcal P$  whose objects are path complexes  and 
 whose morphisms are morphisms of path complexes.

Let $J=\{0,1\}$  be the  two-element set.  For any set $V=\{0,\dots, n\}$, let $V\times J$ 
be as usual the Cartesian product.
Let $V^{\prime}$ be a copy of the set $V$ with the elements $\{0^{\prime},\dots, n^{\prime}\}$ 
where $i^{\prime}\in V^{\prime}$ corresponds to $i\in V$. Then we can  identify $V\times J$  
with  $V\coprod V^{\prime}$ in  such 
a way that  
$(i,0)$ corresponds to $i$ and $(i,1)$ corresponds to $i^{\prime}$ for all $i\in V$.
Thus $V$ is identified with $V\times \{0\}\subset V\times J$ and  $V^{\prime}$ is 
identified with $V\times \{1\}\subset V\times J$. 
 The natural bijection $V\cong V^{\prime}$ defines the set of path  $P^{\prime}$ on 
 the set $V^{\prime}$  by the  condition: $(i_0^{\prime}\dots i_n^{\prime})\in P^{\prime}$ iff 
$(i_0\dots i_n)\in P$. We   define  a path complex $\Pi^{\uparrow}=(V\times J, P^{\uparrow})$  where  
\begin{equation*}
\begin{matrix}
P^{\uparrow}=P\cup P^{\prime}\cup P^{\#}, \\
{P^\#}=\{ q^\#_k=(i_0\dots i_k \,i_k^{\prime} , i^{\prime}_{k+1}\dots 
i_n^{\prime})\, | \, q=(i_0\dots i_k \, i_{k+1}\dots i_n)\in P \}.\\
\end{matrix}
\end{equation*}
There are   morphisms 
$
i_{\bullet}\colon \Pi \to \Pi^{\uparrow}
$
and $j_{\bullet}\colon \Pi \to \Pi^{\uparrow}$
which are  induced by the natural inclusion of $V$ in $V\times \{0\}$ and  in   $V\times\{1\}$, respectively.

\begin{definition}\label{d2.1} \rm i) A path complex $\Pi=(V, P)$ is called 
\emph{weighted} if it is equipped with a function $\delta_V \colon V\to R$ where $R$ is a unitary commutative ring. 
This function is called the \emph{weight function}. 
We denote such a path complex  by  $\Pi^{\delta}=(V, P, \delta_V)$.

ii) Let $(V, P,\delta_V)$  and $(W, S, \delta_{W})$  be weighted path complexes. A morphism    
$f\colon (V, P)\to (W,S)$ of   path complexes  is a  \emph{weighted morphism}   if 
$\delta_{W}(f_V(v))=\delta_V(v)$ for every  $v\in V$. We denote such a morphism by $f^{\delta}$.
\end{definition}

Weighted path complexes with  weighted morphisms 
form a category we denote by $\mathcal P^{\delta}$. 

Let $\Pi^{\delta}=(V, P,\delta_V)$  be a weighted path complex. 
Then the path complex $\Pi^{\uparrow}$  has 
the natural structure of the weighted path complex 
${\Pi^{\uparrow}}^{\delta}$ with the weighted
function $\delta_{V\times J}(v,0)=\delta_{V\times J}(v,1)=\delta_{V}(v)$. 

\begin{definition}\label{d2.2} \rm i) We call weighted morphisms 
$f^{\delta}, g^{\delta}\colon \Pi^{\delta}\to \Sigma^{\delta}$  of 
weighted path complexes 
\emph{weighted one-step homotopic}   if 
 morphisms $f,g$  are  one-step homotopic (see \cite[Def. 3.1]{Hyper}) and the  homotopy $F\colon { \Pi^{\uparrow}}\to \Sigma$ 
 is a weighted morphism.  

ii)  Two weighted morphisms  $f^{\delta}, g^{\delta}\colon \Pi^{\delta}\to \Sigma^{\delta}$  
of weighted path complexes   are \emph{homotopic} $f^{\delta}\simeq  f^{\delta}$  if there 
exists a sequence of weighted morphisms 
$$
f^{\delta}=f_0^{\delta},f_1^{\delta},\dots, f_n^{\delta}=g^{\delta}\colon \Pi^{\delta}\to \Sigma^{\delta}
$$ 
such that any two consequent morphisms are one-step homotopic. 
\end{definition}

Thus we obtain   homotopy category $h{\mathcal P^{\delta}}$ whose objects are weighted 
path complexes and  morphisms 
 are classes of weighted homotopic weighted  morphisms.  

\begin{proposition}\label{p2.3} Let 
	$f^{\delta}, g^{\delta}\colon  \Pi^{\delta}\rightarrow \Sigma^{\delta}$ be weighted 
	morphisms of weighted path complexes which are one-step homotopic  as morphisms of 
	path complexes.    Then   morphisms $f^{\delta}$ and $g^{\delta}$ are weighted  
one-step homotopic. 
\end{proposition}
\begin{proof} It follows from the definition of the weight function 
$\delta_{V\times J}$. 
\end{proof}

Let $\Pi^{\delta}=(V,P,\delta_V)$.  We define the weighted  path homology groups 
$H^{\mathbf w}_*(\Pi^{\delta})$ with coefficients in 
a   unitary   commutative ring $R$  similarly to the regular  path homology groups 
\cite[Sec. 2]{Hyper}.   For $n\geq 0$, let  $\mathcal R_n^{reg}(V)$  be a module  generated by all  paths 
$(i_0\dots i_n)$  on the set $V$ which are \emph{regular}, that is $i_k\ne i_{k+1}$ for $0\leq k \le n-1$ and  
$\mathcal R_{-1}^{reg}(V)=0$. Let  $\mathcal R_n^{reg}(P)\subset \mathcal R_{n}^{reg}(V)$ be submodules 
generated by paths from $P$ for $n\geq 0$ and $\mathcal R_{-1}^{reg}(P)=0$.  
We define  a \emph{weighted boundary homomorphism} 
$$
\partial_n^w:\mathcal R_{n}(V)\rightarrow \mathcal R_{n-1}(V)
$$
by putting $\partial_{0}^{w}=0$ and  
$$
\partial^w_n( e_{i_{0}...i_{n}})=\sum_{s=0}^{n}
\left( -1\right)^s\delta_V(i_s)
e_{i_{0}...\widehat{i_{s}}...i_{n}} 
$$
 for  $n\geq 1$, where  $\widehat{i_{s}}$ means omission of the index $i_{s}$. 
Then $\left(\partial^{w}\right)^2=0$ \cite[\S 3]{MiHH}. 
Let us define  $R$-modules $\Omega _{*}^{w}=\Omega _{*}^{w}(\Pi^{\delta})$ 
by setting: $\Omega^w_{-1}=0$ and  
$$
\Omega _{n}^{w}=\left\{ v\in \mathcal{R}^{reg}
_{n}(P)|\partial^w v\in \mathcal{R}_{n-1}^{reg}(P)\right\}\  \text{for}\  n\geq 0.  
$$
Modules   $\Omega_n^w$ with  the differential induced by  $\partial^w $   form  a chain
 complex.
Homology groups $H_*^{\mathbf w}(\Pi^{\delta})$ of this complex 
 are called  \emph{weighted path homology
groups}.  

Recall that a digraph $V=(V_G, E_G)$ is called \emph{weighted} if it is equipped a weight function $\delta_{V_G}\colon V\to R$.
Let $G^{\delta}=(V_G, E_G, \delta_{V_G})$ be a  weighted digraph without  
loops. We define a weighted path complex  
$\mathfrak D^{w}(G)=(V_G, P_G, \delta_{V_G})$ in which  
 a path $(i_0\dots i_n)$ on the set $V_G$ lies in $P_G$ iff  for $k=1,\dots, n$  every pair of consequent vertices
$i_{k-1}, i_k$ lies in  the edge 
$(i_{k-1}\to i_k) \in E_G$. Thus one gets a functor $\mathfrak D^{w}$
from the category of  weighted digraphs $\mathcal D^{\delta}$   to the category of  
weighted path complexes 
$\mathcal P^{\delta}$.  We define the \emph{weighted path homology} of the weighted digraph by 
$$
H_*^{\mathbf w}(G^{\delta})\colon = H_*^{\mathbf w}\left(\mathfrak D^w\left(G^{\delta}\right) \right).
$$

\begin{example}\label{e2.4} \rm  We give two examples of weighted path complexes 
	$\Pi^{\delta}$ for which    the regular  homology groups $H_*(\Pi)$ and 
	$H_*^{\mathbf w}(\Pi^{\delta})$  with coefficients in $\mathbb Z$ are non-isomorphic.

i) Let
 $\Pi^{\delta}=(V,P, \delta_V)$ with
$V=\{a,b,c,d\}, \, P=\{a,b,c,d, ac,ad,bc,bd\}$ and
$\delta_V(a)=\delta_V(b)=1, \, \delta_V(c)=\delta_V(d)=0$. 
The path homology groups of 
$\Pi=(V,P)$ coincide with the regular path homology groups 
of the digraph
$$
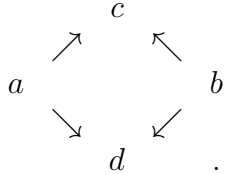
\begin{matrix}
&& c&&\\
  &\nearrow&&\nwarrow&\\
a& &   && b\\
  &\searrow&&\swarrow&\\
  &&d&& .\\
\end{matrix}
$$ 
We have 
$$H_i(\Pi)=\begin{cases} \mathbb Z& \text{for} \  i=0,1\\
                                0& \text{for} \  i\geq 2 
\end{cases}
$$ 
and 
$$
H_i^{\mathbf w}(\Pi^{\delta})=\begin{cases} \mathbb Z^2& \text{for} \  i=0,1\\
                                0& \text{for} \  i\geq 2. 
\end{cases}
$$ 

ii) Let
$\Pi^{\delta}=(V,P, \delta_V)$ with
$
V=\{a,b\}, \, P=\{a,b\}, 
$
and 
$ 
\delta_V(a)=2, \delta_V(b)= 4
$.
The path homology groups of 
$\Pi=(V,P)$ coincide with the regular path homology groups 
of the digraph $a\to b$. We have 
$$
H_i(\Pi)=\begin{cases} \mathbb Z& \text{for} \  i=0\\
                                0& \text{for} \  i\geq 1 
\end{cases}
$$ 
and 
$$
H_i^{\mathbf w}(\Pi^{\delta})=\begin{cases} \mathbb Z\oplus \mathbb Z_2& \text{for} \  i=0\\
                                0& \text{for} \  i\geq 1. 
\end{cases}
$$
\end{example}

Let us fix a weighted path complex  $\Pi^{\delta}=(V,P,\delta)$ with 
a weighted function $\delta\colon V\to R$  such that $\delta(i)$ is 
an invertible element for all $i\in V$.
We define a function $\gamma\colon V\to R$ by putting $\gamma(i)=(\delta(i))^{-1}$. 
Recall that for the weighted path complex 
 ${\Pi^{\uparrow}}^{\delta}=\left(V\times J, P^{\uparrow}, \delta_{V\times J}\right)$, 
we identify $V\times \{0\}$ with $V$,  
$V\times \{1\}$ with $V^{\prime}$   and, hence,  
$\delta_{V\times J}(v, 0)=\delta_{V\times J} (v,1)=\delta_V(v)=\delta_V(v^{\prime})$ for $v\in V$. 
 Thus, for a path $e_{i_0\dots i_n}\in \mathcal R_n^{reg}(P)$ and $0\leq k \leq n$,  we have 
$\delta(i_k)=\delta(i_k^{\prime}), 
\gamma(i_k)=\gamma(i_k^{\prime})
$. For $n\geq 0$, we define a homomorphism 
$
\tau\colon \mathcal R_n^{reg}(V)\to \mathcal R_{n+1}^{reg}(V\times J)
$
 on  basic  regular  $n$-paths
$v=e_{i_{0}...i_{n}}$  by 
\begin{equation}\label{2.1}
\tau(v)=\sum_{k=0}^{n}\gamma(i_k)\left( -1\right) ^{k}e_{i_{0}...i_{k}i_{k}^{\prime
}...i_{n}^{\prime }}.
\end{equation}
  We set $\mathcal R_{-1}^{reg}(V)=0$ and define $\tau=0\colon \mathcal R_{-1}^{reg}(V)\to \mathcal R_0^{reg}(V\times J)$.   

\begin{lemma}
\label{l2.5} For $n\geq 0$ and  any path $v\in \mathcal R _{n}^{reg}(V)$,   we have 
\begin{equation}\label{2.2}
\partial^w \tau(v)+\tau(\partial^w  v)=v^{\prime}- v.
\end{equation}
 \end{lemma}
\begin{proof} It is sufficient
	 to prove the statement for basis elements 
$v=e_{i_{0}\dots i_{n}}$. For $n=0$ and $v=e_{i_0}\in 
\mathcal R^{reg}_0(V)$,
 we have 
$
\partial^w\tau(e_{i_0})=\partial^w(\gamma(i_0)e_{i_0i_0^{\prime}})= \gamma(i_0)\cdot \left(\delta({i_0})e_{i_0^{\prime}}-\delta({i_0}^{\prime})e_{i_0}\right)
=\gamma(i_0)\cdot \delta({i_0})[e_{i_0^{\prime}}-e_{i_0}]=
e_{i_0^{\prime}}-e_{i_0},  \  \tau( \partial^w v)=\tau(0)= 0,
$
and so (\ref{2.2})  is true.  Let  
$v=e_{i_{0}\dots i_{n}}\in \mathcal R^{reg}_{n}(V)$ with $n\geq 1$.  Then 
\begin{eqnarray*}
\partial^w (\tau(v)) &=&\partial^w \left( \sum_{k=0}^{n}( -1) ^{k}\gamma(i_k)e_{i_{0}\dots i_{k}i_{k}^{\prime}\dots  i_{n-1}^{\prime}i_{n}^{\prime }}\right) \\ 
&=& \sum_{k=0}^{n}(-1)^k\left(\sum_{m=0}^{k}(-1)^m
\delta(i_m)\gamma(i_k)e_{i_{0}\dots \widehat{i}_{m}\dots i_ki_{k}^{\prime}\dots i_{n}^{\prime }}\right)\\
&+& \sum_{k=0}^{n}(-1)^k \left(\sum_{m=k}^{n}(-1)^{m+1}\delta(i_m^{\prime})\gamma(i_k)
e_{i_{0}\dots  \dots i_ki_{k}^{\prime}\dots\widehat{i}_{m}^{\prime} \dots i_{n}^{\prime }}\right)\\
&=& \sum_{0\leq m\leq k\leq n}(-1)^{k+m}\delta(i_m)\gamma(i_k)
e_{i_{0}\dots \widehat{i}_{m}\dots i_ki_{k}^{\prime}\dots i_{n}^{\prime }}\\
&+& \sum_{0\leq k\leq m\leq n}(-1)^{k+m+1}\delta(i_m^{\prime})\gamma(i_k)
e_{i_{0}\dots  \dots i_ki_{k}^{\prime}\dots\widehat{i}_{m}^{\prime} \dots i_{n}^{\prime }}\  
\end{eqnarray*}
and
\begin{eqnarray*}
\tau (\partial^w v) &=&\tau \left( \sum_{m=0}^{n}
( -1) ^{m}\delta(i_m)
e_{i_{0}\dots \widehat{i}_{m}\dots  i_{n}}\right) \\ 
&=& \sum_{m=0}^{n}(-1)^m\left(\sum_{k=0}^{m-1}(-1)^k
\delta(i_m^{\prime})\gamma(i_k)e_{i_{0}\dots i_{k}i_{k}^{\prime} \dots  \widehat{i}_{m}^{\prime}\dots  i_{n}^{\prime }}\right)\\
&+ &\sum_{m=0}^{n}
( -1) ^{m}\left(\sum_{k=m+1}^{n}(-1)^{k-1}\delta(i_m)\gamma(i_k)
e_{i_{0}\dots  \dots \widehat{i}_{m} \dots i_ki_{k}^{\prime }\dots i_n^{\prime}}\right)\\
&=& \sum_{0\leq k<  m\leq n}(-1)^{k+m}\delta(i_m^{\prime})\gamma(i_k)
e_{i_{0}\dots i_{k}i_{k}^{\prime} \dots  \widehat{i}_{m}^{\prime}\dots  i_{n}^{\prime }}\\
&+ &\sum_{0\leq m<  k\leq n}(-1)^{k+m-1}\delta(i_m)\gamma(i_k)
e_{i_{0}\dots  \dots \widehat{i}_{m} \dots i_ki_{k}^{\prime }\dots i_n^{\prime}}\ .
\end{eqnarray*}
Hence
\begin{eqnarray*}
\partial^w \tau(v)+\tau(\partial^w  v)&=&\sum_{0\leq k\leq n}(-1)^{k+k}\underbrace{\delta(i_k)\gamma(i_k)}_{=1}
e_{i_{0}\dots \widehat{i}_{k}i_{k}^{\prime}\dots i_{n}^{\prime }}\\
&+ &\sum_{0\leq k\leq n}(-1)^{k+k+1}\underbrace{\delta(i_k^{\prime})\gamma(i_k)}_{=1}
e_{i_{0}\dots  \dots i_k \widehat{i}_{k}^{\prime}\dots i_{n}^{\prime }}\\
&=&\sum_{0\leq k\leq n}
e_{i_{0}\dots i_{k-1}i_{k}^{\prime}\dots i_{n}^{\prime }}- \sum_{0\leq k\leq n}
e_{i_{0}\dots  \dots i_k {i}_{k+1}^{\prime}\dots i_{n}^{\prime }}\\
&=&e_{i_0^{\prime}\dots i_n^{\prime}}+\sum_{1\leq k\leq n}
e_{i_{0}\dots i_{k-1}i_{k}^{\prime}\dots i_{n}^{\prime }}\\ 
&-&\sum_{0\leq k\leq n-1}
e_{i_{0}\dots  \dots i_k {i}_{k+1}^{\prime}\dots i_{n}^{\prime }}-e_{i_0\dots i_n}\\
&=&e_{i_0^{\prime}\dots i_n^{\prime}}-e_{i_0\dots i_n }\\
&+&\left( \sum_{0\leq k-1\leq n-1}
e_{i_{0}\dots i_{k-1}i_{k}^{\prime}\dots i_{n}^{\prime }}- \sum_{0\leq k\leq n-1}
e_{i_{0}\dots  \dots i_k {i}_{k+1}^{\prime}\dots i_{n}^{\prime }}\right)\\
&=&e_{i_0^{\prime}\dots i_n^{\prime}}-e_{i_0\dots i_n }= v^{\prime}-v.
\end{eqnarray*} 
\end{proof}

\begin{theorem}\label{t2.6}  Let 
$
f^{\delta}\simeq  g^{\delta}\colon\Pi^{\delta}=(V,P, \delta_V)\to (W, S, \delta_W) =\Sigma^{\delta}
$
be weighted homotopic morphisms of weighted path complexes such that the elements 
$\delta_V(i)\in R$ are invertible for all $i\in V$.  Then morphisms $f^{\delta}$ 
and $g^{\delta}$ induce  the chain homotopic 
morphisms of  weighted chain complexes 
$$
f_*^{\delta}\simeq g_*^{\delta}\colon \Omega_*^w(P)\to \Omega_*^w(S) 
$$
and  hence,  the  equal    homomorphisms of weighted path homology groups.
\end{theorem} 
\begin{proof}  It is sufficient to  consider the case of one-step weighted homotopy 
	 $F\colon {\Pi^{\uparrow}}^{\delta}\to \Sigma^{\delta}$  between $f^{\delta}$
	  and $g^{\delta}$.  By definition 
$
F_{*}\circ {[i_{\bullet}]}_*=f_{*}$ and
$F_{*}\circ {[j_{\bullet}]}_*=g_{*}$.
We define  a chain
homotopy 
$
L_{n}\colon \Omega _{n}^w(P)\rightarrow \Omega _{n+1}^w(S)
$
such that
\begin{equation*}
\partial^w L_{n}+L_{n-1}\partial^w =g_{\ast }-f_{\ast }.
\end{equation*}
Let the element $v\in \mathcal R_n^{reg}(P)$ belong to $\Omega_n^w(P)$. 
Then,  by definition,  $\tau(v)\in  \mathcal R_{n+1}^{reg}(P^{\uparrow})$.
To prove  that 
$\tau(v)\in \Omega_{n+1}^w(P^{\uparrow})$,  it is sufficient to check that $\partial^w\tau(v)\in 
\mathcal R_{n}^{reg}(P^{\uparrow})$.  By  definition,     
$\partial^w v\in \mathcal R_{n-1}^{reg}(P)\subset \Omega_{n-1}^w(V)$.   By   Lemma \ref{l2.5}
$
\partial^w\tau(v)= -\tau(\partial^w v)+v^{\prime}-v
$
where the right summands belongs to $\mathcal R_n^{reg}(P^{\uparrow})$. Hence, 
 $\tau(v)\in \Omega_{n+1}^w(P^{\uparrow})$.
For $n\geq 0$, we  define  homomorphisms  
$$
L_{n}(v)\colon \Omega_n^w(P)\to \Omega_{n+1}^w(S) 
$$
by $L_n(v)\colon= F_{\ast }\left( \tau(v)\right)$.  
Now, we check that $L_n$ is a chain homotopy:
\begin{eqnarray*}
(\partial^w L_{n}+L_{n-1}\partial^w )(v) &=&\partial^w (F_{\ast }(\tau (v) 
))+F_{\ast }(\tau(\partial^w v))  \\
&=&F_{\ast }\left( \partial^w \tau(v)\right) +F_{\ast }(\tau(\partial^w v)) \\
&=&F_{\ast }(\partial^w \tau(v)+\tau(\partial^w v)) \\
&=&F_{\ast }\left( v^{\prime }-v\right)
=g_{\ast }^{\delta}\left( v\right) -f_{\ast }^{\delta}\left( v\right) .
\end{eqnarray*}
\end{proof}

 \section{Path homology  of weighted directed  hypergraphs}\label{S3}
\setcounter{equation}{0}

\begin{definition} \label{d3.1}\rm   i) A \emph{directed hypergraph} $G=(V,E)$ consists 
	of a finite  set of \emph{vertices} $V$  and a   set of \emph{arrows}
	$E=\{\mathbf e_1,\dots , \mathbf e_n\}$   where $\mathbf e_i\in E$    
	is an ordered pair $(A_i,B_i)$  of disjoint non-empty subsets  of 
	$V$  such that 
	$V=\bigcup_{\mathbf e_i\in  E} (A_i\cup B_i)$. 
	The set $A=\mbox{orig }(A\rightarrow B)$ is called 
	the \emph{origin} of the arrow and the set  $B=\mbox{end}(A\rightarrow B)$
	is called the \emph{end} of the arrow. The elements of $A$ are called 
	the \emph{initial} vertices of $A\to B$ and  the elements of $B$ are 
	called its \emph{terminal} vertices.  

ii) A directed hypergraph  $G=(V, E)$ is called 
\emph{weighted} if it is equipped with a function $\delta_V \colon V\to R$ 
where $R$ is a unitary commutative ring. 
This function is called the \emph{weight function}. 
We denote such a directed hypergraph   by  $G^{\delta}=(V, E, \delta_V)$.
\end{definition}

For a  set  $X$  let ${\mathbf{P}}(X)$,  denote as usual its power set. 
We define  a set 
$
\mathbb P(X)\colon =  \{{\mathbf{P}}(X)\setminus \emptyset\}\times  
\{{\mathbf{P}}(X)\setminus \emptyset\}.
$  
Every  map  $f\colon V\to W$ induces a map 
$\mathbb P(f)\colon\mathbb P(V)\to \mathbb P(W)$. 
For a directed hypergraph $G=(V,E)$,   by Definition \ref{d3.1}, 
we have the natural map  $\varphi_G\colon E\to \mathbb P(V)$ 
defined by  $\varphi_G(A\to B)\colon =(A,B)$. 
Let $X\subset V$ be a subset of the set of vertices of a 
weighted directed hypergraph $G^{\delta}$. 
Define  the \emph{weight}  $|X|$ of a set $X$ by setting 
$$
|X|\colon = \sum\limits_{x\in X}\delta_V(x).
$$

Let $G=(V_G,E_G)$  and
$H=(V_H,E_H)$  be two  directed hypergraphs.
By \cite[Def. 2.1]{Miarx},  
the  morphism    
$f\colon G\to H$   is given by a pair of maps  $f_V\colon V_G\to V_H$ 
and $f_E\colon E_G\to E_H$ such that 
the following diagram 
\begin{equation}\label{3.1}
\begin{matrix}
E_G&\overset{\varphi_G}\longrightarrow &\mathbb P(V_G)\\
\ \ \downarrow f_E&&\ \  \downarrow \mathbb P(f_V)\\
E_H&\overset{\varphi_H}\longrightarrow &\mathbb P(V_H)\\
\end{matrix}
\end{equation}
is commutative.  Let $G^{\delta}=(V_G,E_G, \delta_{V_G})$  and
$H^{\delta}=(V_H,E_H, \delta_{V_H})$  be two weighted  
directed hypergraphs and let   $f\colon G\to H$   be a morphism. 

\begin{definition}\label{d3.2} \rm   i) The  morphism $f$ is called
\emph{vertex-weighted}  
if $\delta_{V_H}(f_V(v))=\delta_{V_G}(v)$ for every  $v\in V_G$. 
We denote such morphism by $f^{v}$.

ii)   The morphism $f$  is called
\emph{edge-weighted}  
if  for every $(A\to B)\in E_G$ and 
$f_E(A\to B)=A^{\prime}\to B^{\prime}$ the conditions  
$|A|=|A^{\prime}|, \ |B|=|B^{\prime}|$ are satisfied.

iii)  The  morphism $f$ is  called \emph{strong-weighted}  
if it is  \emph{vertex-weighted}  and \emph{edge-weighted}  
simultaneously. We denote such a morphism by $f^{s}$. 
\end{definition}
Thus,  we obtain categories     $\mathcal{DH}^{v}$,   
$\mathcal{DH}^{e}$, and  $\mathcal{DH}^{s}$ of  
vertex-weighted, edge-weighted, and  strong-weighted 
directed hypergraphs, respectively.

Let $G=(V_G,E_G)$ be a directed hypergraph.
We define 
subsets 
${\mathbf{P}}_0(G)$, ${\mathbf{P}}_1(G)$, and 
$\mathbf P_{01}={\mathbf{P}}_{0}(G)\cup {\mathbf{P}}_{1}(G)$ of 
${\mathbf{P}}(V_G)\setminus \emptyset$ by
setting
\begin{equation*}\label{3.0}
\begin{matrix}
{\mathbf{P}}_0(G)=\{A\in  {\mathbf{P}}(V_G)\setminus \emptyset| \exists 
B\in  {\mathbf{P}}(V_G)\setminus \emptyset:  \ A\to B\in E_G\}, \\
{\mathbf{P}}_1(G)=\{B\in  {\mathbf{P}}(V_G)\setminus \emptyset| \exists 
A\in  {\mathbf{P}}(V_G)\setminus \emptyset: \ A\to B\in E_G\}. \\
\end{matrix}
\end{equation*} 

We define a digraph 
$
\mathfrak N(G)= (V^n_G, E^n_G)
$
where  
$
V^n_G=\{ C\in {\mathbf{P}}(V)\setminus \emptyset | C\in {\mathbf{P}}_{01}(G)\}
$
  and   
$
E^n_G=\{A\to B| (A\to B)\in E\}.
$
 Note that   $\mathfrak N$ is a functor  from  the category $\mathcal{DH}$ 
of directed hypergraphs  \cite{Miarx}
to  the category  $\mathcal D$  of  digraphs.
 We equip the digraph $\mathfrak N(G)=(V^n_G, E^n_G)$  with a structure of a
\emph{weighted digraph} by defining   a weight  function   in the following way 
$\delta_{V^n_G} (A)\colon \overset{def}=|A|$  where $|A|$ 
is the weight of the set   $A$. We  denote such a weighted digraph by $\mathfrak N^{w}(G)$. 

\begin{proposition}\label{p3.3} Every edge-weighted morphism of  weighted directed hypergraphs 
$
f^e\colon G^{\delta}\to H^{\delta}
$
defines a morphism  of weighted digraphs 
$$
[\mathfrak N^{w} (f)]=(f_{V}^n, f_{E}^n) \colon (V_G^n,E^n_G,  \delta_{V^n_G})\to (V_H^n, E^n_H, \delta_{V_H^n} )
$$ 
by
 $
f^n_V(C) \colon =[{\mathbf{P}}(f)]\circ \phi_G(C)
$
and 
$
f_E^n(A\to B) =(f_V(A)\to f_V(B))\in E^n_H
$.  Thus, we obtain a functor  
$\mathfrak N^{ w} $  from the category $\mathcal{DH}^{e}$ of edge-weighted directed hypergraphs 
to the category  $\mathcal D^{\delta}$  of weighted  digraphs.
\end{proposition}

\noindent
The composition  $\mathfrak D^{w}\circ \mathfrak N^{w}$  
is  a functor $\mathcal{DH}^{e} \to \mathcal P^{\delta}$,  and  we define the
 \emph{edge-weighted path homology groups} of the  weighted 
directed hypergraph 
 $G^{\delta}$ by
$$
H_*^{\mathbf e}(G)\colon = H_*^{\mathbf w}(\mathfrak D^w\circ \mathfrak N^{w}(G) ).
$$

Now we use  the notion of a \emph{line digraph} 
$I_n=(V_{I_n}, E_{I_n}) \ (n\geq 0)$  from \cite[Sec. 3.1]{MiHomotopy}. 
Note that we have two digraphs $I_1$, namely $0\to 1$ and $1\to 0$.   
For a digraph $G$  the \emph{box product} 
 $G\Box I_n$ is a digraph  with  
$V_{G\Box I_n}=V_{G}\times V_{I_n}$ and with the set of arrows
 $E_{G\Box I_n}$ such that there is an arrow 
 $((x,i)\rightarrow (y,j))\in E_{G\Box H}$  if and only 
if either $x=y$ and  $i\rightarrow j$  or  $x\rightarrow y$  and 
$i=j$, see \cite{MiHomotopy}. If $G^{\delta}$ is a weighted digraph,  
we  equip the digraph $G\Box I_n$ with the structure of a weighted 
digraph by setting $\delta_{V_{G\Box I_n}}(v, i)=\delta_{V_G}(v)$
for $v\in V_G$, $i\in V_{I_n}$. We denote this weighted digraph by 
$G^{\delta}\Box I_n$. 

For a  directed hypergraph $G$,   the  \emph{box product} 
  $G\Box I_n=(V_{G\Box I_n}, E_{G\Box I_n})$ 
  is a directed hypergraph with 
$V_{G\Box I_n}=V_{G}\times V_{I_n}$ and the set of arrows $E_{G\Box I_n}$ being a 
 union of   
$
\{A\times i\to  B\times i\}
$,
 $(A\to B)\in E_G$,   $i\in V_{I_n}$ and 
$ \{A\times i\to  A\times j\}$, 
$ (i\to j)\in E_{I_n}$,  $A\in 
{\mathbf{P}}_{01}(G)$. If $G^{\delta}$ is a weighted directed 
hypergraph,  we equip the directed hypergraph $G\Box I_n$ 
with  the structure of a weighted directed hypergraph 
 by setting $\delta_{V_{G\Box I_n}}(v\times  i)=\delta_{V_G}(v)$
for $v\in V_G$, $i\in V_{I_n}$. We denote this 
weighted directed hypergraph by $G^{\delta}\Box I_n$.  
Recall, see e.g.  \cite{Miarx},  that two  morphisms  
$f_0, f_1\colon G\rightarrow H$ of directed hypergraphs  
	are called \emph{one-step homotopic} 
	$f_0\simeq_1 f_1$,  if there exists 
	a morphism  
	$F\colon G\Box I_1\rightarrow H$,  such that 
\begin{equation*}
F|_{G\Box \{0\}}=f_0\colon G\Box \{0\}\rightarrow H,\ \ F|_{G\Box
\{1\}}=f_1\colon G\Box \{1\}\rightarrow H.
\end{equation*} 

Two  morphisms  $ f,g\colon G \to  H$  of directed 
hypergraphs are called  \emph{homotopic}  
$ f\simeq  g$,  if there exists a sequence 
of morphisms 
$
f_i\colon G\to H$ for   $i=0,\dots,  n$  such that 
 $ f=f_0\simeq_1  f_1\simeq_1 \dots \simeq_1  f_n =g$.

\begin{proposition}\label{p3.4} Let   
	$f_0^{\chi}, f_1^{\chi}\colon G^{\delta}\rightarrow H^{\delta}$   
	be weighted morphisms of weighted directed hypergraphs  where
	 $\chi$ is $v$, $e$, or $s$.  
If $f_0$ and $f_1$ are one-step homotopic as morphisms of directed 
hypergraphs by means of  a homotopy $F$,  then   $F$ can be considered as a weighted morphism 
$F^{\chi}\colon G^{\delta}\times I_1\to H^{\delta}$. 
\end{proposition}
\begin{proof}  It follows from the definition  of  $G^{\delta}\Box I_1$.   
\end{proof}

\noindent If $\chi$ means one of the values  $v$, $e$, or $s$  we obtain  
a \emph{homotopy category of $\chi$-weighted directed hypergraphs}  
${h\mathcal {DH}^{\chi}}$ whose objects are  weighted directed hypergraphs and morphisms 
are  classes of homotopic $\chi$-weighted  morphisms.   

\begin{lemma}\label{l3.5}  Let  $I_1=(0\to 1)$  and $G^{\delta}$  be a weighted digraph. 
	Then there is an equality   
$
{[\mathfrak D^w(G^{\delta})]^{\uparrow}}^{\delta}=
\mathfrak D^w(G^{\delta}\Box I_1)
$
of weighted path complexes.
\end{lemma}
\begin{proof}  It follows from the definitions of the functor $\mathfrak D^w$, the box product, 
	and the path complex $\Pi^{\uparrow}$.   
\end{proof}

\begin{lemma}\label{l3.6}  Let $G^{\delta}$ be a weighted directed hypergraph and 
$I_1=(0\to 1)$. Then there is an equality   
$
{[\mathfrak D^w\mathfrak N^{w}(G)]^{\uparrow}}^{\delta}=\mathfrak D^w
\mathfrak{N^w}(G\Box I_1)
$
of the weighted path complexes.
\end{lemma}
\begin{proof}  By Lemma   \ref{l3.5},  $
{[\mathfrak D^w(\mathfrak N^w(G))]^{\uparrow}}^{\delta}=
\mathfrak D^w\left[\mathfrak N^w(G)\Box I_1\right]$. The weighted digraphs 
$\mathfrak N^w(G)\Box I_1$ and $\mathfrak N^w(G\Box I_1)$ are equal by definition of the 
the box product and the functor $\mathfrak N^w$. Hence 
$\mathfrak D^w\left[\mathfrak N^w(G)\Box I_1\right]=\mathfrak D^w\mathfrak N^w(G\Box I_1)
$ and the result follows.
\end{proof}

\begin{theorem}\label{t3.7} Let $f^e, g^e\colon G^{\delta}\to 
H^{\delta}$ be  
edge-weigh\-ted homotopic morphisms 
	of weighted  directed hypergraphs and assume that all elements of the set 
$\mathbb K=\{k\in R| k=|A|,   A\in \mathbf P_{01}(G)\}$
 are invertible in $R$. Then the induced homomorphisms 
$$
f_*, g_*\colon H^{\mathbf e}_*(G, R)\to H^{\mathbf e}_*(H, R)
$$ 
coincide.  
\end{theorem}
\begin{proof} It follows from Lemma \ref{l3.6} and Theorem \ref{t2.6}. 
\end{proof}

Let us consider a subcategory $h\mathcal{D}\mathcal{H}^1$ of the category 
$h\mathcal{D}\mathcal{H}^{e}$  in which $\delta_V\equiv 1$ for every 
weighted directed hypergraph. 

\begin{theorem}\label{t3.8} If 
	the ring $R$ is an algebra over rationals, then the weighted homology groups 
$ H^{\mathbf e}_*(-, R)$  are homotopy invariant on the category 
$h\mathcal{D}\mathcal{H}^1$. 
\end{theorem}
\begin{proof} Similar to the proof of  Theorem \ref{t3.7}. 
\end{proof}
\smallskip

Now,  we describe  several weighted  path homology theories on the  category 
$\mathcal{DH}^v$ which are similar to the path homology theories on  
$\mathcal{DH}$ constructed in  \cite[Sec. 3]{Miarx}. 
We present here only functorial and homotopy invariant theories. 
The other path homology theories from \cite{Miarx} can be transferred to 
the case of weighted hypergraphs similarly. 
Let us consider  the functors $\mathfrak C$, 
$\mathfrak B$, and the functor given by the composition  
$\mathfrak H^{ q}\circ \mathfrak E$  
from the category $\mathcal{DH}$  to the category  $\mathcal P$ of 
path complexes  \cite{Miarx}. 
We recall   the definition   of these functors for  objects, 
then it will be clear  
for morphisms.  

 Let $G= (V, E)$ be  a  directed hypergraph. 
Define a path complex    $\mathfrak C(G)=(V^c, P_G^c)$  by 
$V^c=V$ and  a path $(i_0\dots i_n)\in P_{V}$ lies in $P_G^c$  iff 
for any pair of consequent vertices $(i_{k}, i_{k+1})$  
either $i_{k}=i_{k+1}$ or
there is an  edge $\mathbf e=(A\to B)\in E$  such that 
$i_k$ is the initial vertex and  $i_{k+1}$ is the terminal vertex of  $\mathbf e$. 

Recall the definition of the \emph{concatenation} $p\vee q$ of two paths 
$p=(i_0,\dots, i_n)$ and $q=(j_0, \dots,j_m)$ on a set $V$. 
  The concatenation is well defined only for $i_n = j_0$ and in such a case,  it  is a path on $V$ given by 
$p\vee q= (i_0 \dots  i_nj_1 \dots j_m)$.  Define a path complex  $\mathfrak B(G)=(V^b_G, P_G^b)$   where $V^b_G=V$ and  
a path $q=(i_0\dots i_n)\in P_{V}$ lies in $P_G^b$  iff  there is a sequence of 
 edges $(A_0\to B_0), \dots, (A_r\to B_r)$  in $E$ such that   
 $B_i\cap A_{i+1}\ne \emptyset$ for $0\leq i\leq r-1$ and  the path $q$ can be presented in the form 
\begin{equation*}
\left(p_0\vee v_0w_0 \vee p_1  \vee v_1w_1\vee p_2\vee \dots  \vee p_{r}\vee  v_r w_r\vee p_{r+1}\right)
\end{equation*}
where $p_0\in P_{A_0}$,  $p_{r+1}\in P_{B_r}$,  $v_i\in A_i$, $w_i\in B_i$, $p_i\in P_{B_{i-1}}\cap  P_{A_{i}}$  
for  $1\leq i\leq r$  and all concatenations  are well defined.  

Recall that a   \emph{hypergraph} $G=(V,E)$  consists of  a non-empty    
set  $V$  of vertices and  a  set  of edges $E=\{\mathbf e_1,\dots, \mathbf e_n\}$ 
which are    distinct  subsets of $V$ such that 
	$
	\bigcup_{i=1}^n \mathbf  e_i=V
	$ and every $\mathbf e_i$ contains strictly more than one element.  
For a directed hypergraph $G$,  we define a hypergraph $\mathfrak E(G)=(V^e, E^e)$ 
where  $V^e=V$ and 
\begin{equation*}
E^{e}=\{C\in {\mathbf{P}}(V)\setminus \emptyset \, | \, C=A\cup B, (A\to B)\in E\}.
\end{equation*}
For a hypergraph   $G= (V, E)$, 
define a path complex  $\mathfrak H^{2}(G)=(V^{2}, P_G^{2})$ of \emph{density two} 
 where $V^{2}=V$ and   a path $(i_0\dots i_n)\in P_{V}$ lies $\in P_G^{2}$  
 iff  every two  consequent vertices of this path
lie in an  edge $\mathbf e\in E$, see \cite{Hyper}. 
Thus for every directed hypergraph $G$ the composition   
$\mathfrak H^{2}\circ \mathfrak E$ defines a path complex 
$\mathfrak H^{2}\circ \mathfrak E(G)$. 

For a weighted directed hypergraph   $G^{\delta}= (V, E, \delta_V)$, 
the path complexes  $\mathfrak C(G)$, 
$\mathfrak B(G)$, and $\mathfrak H^{2}\circ \mathfrak E(G)$  
have the natural structure of weighted path complexes given by 
the weight function $\delta_V\colon V\to R$. We denote  weighted 
path complexes by   $\mathfrak C^{v}(G^{\delta})$, 
$\mathfrak B^{v}(G^{\delta})$, and $[\mathfrak H^{2}\mathfrak E(G^{\delta})]^v$, 
respectively.

\begin{proposition}\label{p3.9} If $f^{v}\colon G^{\delta}\to H^{\delta}$ 
is a vertex-weighted morphism of weighted directed hypergraphs, then the morphisms
$$
\begin{matrix}
\mathfrak C(f)\colon  \mathfrak C(G)\to  \mathfrak C(H),\\
\mathfrak B(f)\colon  \mathfrak B(G)\to  \mathfrak B(H),\\
\mathfrak H^{2} \mathfrak E(f)\colon  \mathfrak H^{2} 
\mathfrak E (G)\to  \mathfrak H^{2} \mathfrak E(H)
\end{matrix}
$$
 are vertex weighted  for the weighted function $\delta_V$. 
 Thus,  we  have the functors $\mathfrak C^{v}$, 
$\mathfrak B^{v}$, and $[\mathfrak H^{2} \mathfrak E]^v$ from  
$\mathcal{DH}^{v}$ to $\mathcal P^{\delta}$. 
\end{proposition}
\begin{proof} It follows from the definition of morphisms in the category 
$\mathcal{DH}^v$ and from the definition of the weight function on the objects of
corresponding  categories.
\end{proof}

For any  weighted directed hypergraph $G^{\delta}=(V,E,\delta_V)$,   we call
 \begin{equation}\label{3.2}
\begin{matrix}
 H_*^{\mathbf {c/v}}(G^{\delta})\colon = H_*^{\mathbf w}(\mathfrak C^{v}(G^{\delta})), \\
 H_*^{\mathbf {b/v}}(G^{\delta})\colon = H_*^{\mathbf w}(\mathfrak B^{v}(G^{\delta})), \\
 H_*^{2/v}(G^{\delta})\colon = H_*^{\mathbf w}
([\mathfrak H^{2} \mathfrak E]^v(G^{\delta})), \\
\end{matrix}
\end{equation}
 as \emph{weighted path homology groups},  \emph{weighted bold path homology groups}, 
 and \emph{weighted non-directed path homology groups of density two}, respectively. 

Now we  reformulate   Lemma 3.5, Lemma 3.10, and  Lemma 3.16 of  
\cite{Miarx} for the case of functors $\mathfrak C^{v}$, 
$\mathfrak B^{v}$, and $[\mathfrak H^{2} \mathfrak E]^v$ from  
$\mathcal{DH}^{v}$ to $\mathcal P^{\delta}$. Let    $G^{\delta}=(V,E,\delta_V)$ 
be a weighted directed hypergraph and  
$I_1=(0\to 1)$.

\begin{lemma}\label{l3.10}
 i) There is a natural isomorphism
$\mathfrak{C}^v(G^{\delta}\Box I_1)
\cong{[[\mathfrak{C}^v(G^{\delta})]^{\uparrow}]}^{\delta}$
of weighted 
path complexes.

ii) There exists  an inclusion 
$\lambda^v\colon \left[[\mathfrak{B}^v(G^{\delta})]
^{\uparrow}\right]^{\delta}\to \mathfrak{B^v}(G^{\delta}\Box I_1)$
of weighted path complexes. The restrictions of 
$\lambda^v$ to the  images of the morphisms $i_{\bullet}$ and $j_{\bullet}$,   
defined in Section \ref{S2}, 
are the natural identifications. 

iii) There is the  inclusion 
$
\mu^v\colon
 \left[\left[[\mathfrak H^{ 2}\circ\mathfrak{E}]^v(G^{\delta})\right]^{\uparrow}\right]^{\delta}\to 
\left[\mathfrak{H}^{2}\circ \mathfrak E\right]^v(G^{\delta}\Box I_1)
$
of weighted path complexes.
\end{lemma}
\begin{proof} By the definition, the weight functions on the vertex set  
$V\times J$  are given by the function 
$\delta_{V\times J}(v,0)=\delta_{V\times J}(v,1)=\delta_{V}(v)$ and 
hence,   coincide for all path camplexes  above. Now the result follows 
from \cite{Miarx}. 
\end{proof}

Let $H^{\chi/\mathbf{v}}_*(G^{\delta})$ denote one of the homology groups 
$H_*^{\mathbf {c/v}}(G^{\delta})$, 
 $H_*^{\mathbf {b/v}}(G^{\delta})$, 
 $H_*^{2/v}(G^{\delta})$ defined in (\ref{3.2}). 

\begin{theorem}\label{3.11} Let $f^v, g^v\colon 
G^{\delta}_1=(V_1,E_1,\delta_{V_1})\to 
G^{\delta}_2=(V_2, E_2, \delta_{V_2})$ be  
vertex-weigh\-ted homotopic morphisms 
	of weighted  directed hypergraphs and assume that all elements of the set 
$\mathbb K=\{k\in R| k=\delta_{V_1}(v),   v\in V_1\}$
 are invertible in $R$. Then the induced homomorphisms 
$$
f_*^v, g_*^v\colon H^{\chi/\mathbf{v}}_*(G_1^{\delta}, R)\to G^{\chi/\mathbf{v}}_*(G_2^{\delta}, R)
$$ 
coincide.  
\end{theorem}
\begin{proof} It follows from Lemma \ref{l3.10} and Theorem \ref{t2.6}.
\end{proof}

All data generated or analysed during this study are included in this published article.

Yuri Muranov, 
Faculty of Mathematics and 
Computer Science
University  of Warmia and Mazury in Olsztyn, 
S\l oneczna 54, 
10-710 Olsztyn, Poland 

 E-mail: muranov@matman.uwm.edu.pl
\smallskip

Anna Szczepkowska, Faculty of Mathematics and 
Computer Science
University of Warmia and Mazury in Olsztyn, 
S\l oneczna 54,
10-710 Olsztyn, Poland 

 E-mail: anna.szczepkowska@matman.uwm.edu.pl
\smallskip

Vladimir Vershinin, 
D\'epartement des Sciences Math\'ematiques
Universit\'e de Montpellier,  
Place Eug\'ene Bataillon,
34095 Montpellier
cedex 5, France

E-mail: vladimir.verchinine@univ-montp2.fr

\begin{thebibliography}{20}

\bibitem{Sur}
Ausiello, G.,  Laura, L.: Directed hypergraphs: Introduction and
  fundamental algorithm --- a survey. Theoretical Computer Science.
  \textbf{658}(Part B), 293 -- 306  (2017)


\bibitem{Bat}
 Battiston, F., Cencetti, G.,  Iacopini, V. Latora,  M. Lucas,  M., Patania, A.,   Young J.-G., 
G. Petri, G.: Networks beyond pairwise interactions.  Structure and dynamics. Physics
Reports. \textbf{874},   1 -- 92 (2020)

\bibitem{Berge}
 Berge, C.: Graphs and hypergraphs. North-Holland Publishing Company, Amsterdam, London; American Elsevier Publishing Company, Inc., New-York (1973)

\bibitem{Embed}
 Bresson, S.,  Li, J.,  Ren, S.,  Wu, J.: The embedded
  homology of hypergraphs and applications. Asian Journal of Mathemarics.
  \textbf{23} (3),   479 -- 500 (2019) 

\bibitem{Graham}
 Chung, F.P.K.,   Graham,   R.L.: Cohomological aspects of hypergraphs.
  Trans. Amer. Math. Soc. \textbf{334}(1), 365 -- 388 (1992)

\bibitem{Betti}
Emtander,  E.: Betti numbers of hypergraphs. Commun. Algebra. \textbf{37}(5),  1545 -- 1571  (2009)

\bibitem{Cang} 
Cang, Z.,  Wei, G. W. Persistent Cohomology for Data With Multicomponent Heterogeneous Information. SIAM J. Math Data Sci. \textbf{2(2)}, 396 -- 418, 
doi: 10.1137/19m1272226 (2020)

\bibitem{Gallo}
Gallo, G.,  Longo, G., Nguyen,   S., Pallottino,  S.: Directed hypergraphs
  and applications,  Disc. Appl. Maths. \textbf{42},  177 -- 201 (1993)


\bibitem{Hyper}
Grigor'yan, A.,  Jimenez, R., Muranov, Y.,  Yau, S.-T.:
Homology of path complexes and hypergraphs. Topology and its
  Applications. \textbf{267},  106 -- 877 (2019).

\bibitem{MiHomotopy}
 Grigor'yan, A., Lin, Y., Muranov, Y.,  Yau, S.-T.:
Homotopy theory for digraphs. Pure and Applied Mathematics Quarterly. 
  \textbf{10}(4) , 619 -- 674 (2014)


\bibitem{Pcomplex} Grigor'yan, A., Lin, Y., Muranov, Y.,  Yau, S.-T.:
Path complexes and their homologies.  Journal of Mathematical
  Sciences.  \textbf{248},  564 -- 599 (2020)

\bibitem{Forum}
Grigor'yan, A.,  Muranov, Y.,  Vershinin, V.,  Yau, S.-T.:
 Path homology theory of multigraphs and quivers.  Forum Math.
  \textbf{30}(5),  1319 -- 1337  (2018)



\bibitem{MiHH} Grigor'yan, A.,  Muranov, Y.,  Yau, S.-T.:
 On a cohomology of digraphs and hochschild cohomology. J.
  Homotopy Relat. Struct. \textbf{11}, 209 -- 230  (2016)



\bibitem{Liu}  Liu, X., Wang, X.,  Wu, J.,  Xia, K.:  Hypergraph-based persistent cohomology (HPC) for molecular representations in drug design. Brief Bioinform.  Sep 2;22(5):bbaa411. doi: 10.1093/bib/bbaa411. PMID: 33480394 (2021)



\bibitem{Miarx}
Muranov, Y., Szczepkowska, A.,  Vershinin, V.: Path homology of
  directed hypergraphs.  arXiv:2109.09842 [math.AT], accepted in HHA 

\bibitem{Parks}
Parks, A.D., Lipscomb,   S.L.: Homology and hypergraph acyclicity: a
  combinatorial invariant for hypergraphs. Naval Surface Warfare Center (1991)

\bibitem{Talbi}
 Talbi, M.E.,  Benayat, D.: Homology theory of graphs.
  Mediterranean J. of Math.  \textbf{11}(2), 813 -- 828 (2014)

\end{thebibliography}
\end{document}